\documentclass[12pt]{amsart}

\usepackage{color}
\usepackage{amstext}
\usepackage{amsfonts}
\usepackage{amsthm}
\usepackage{amsmath}
\usepackage{amssymb}
\usepackage{latexsym}
\usepackage{amsfonts}
\usepackage{graphicx}
\usepackage{texdraw}
\usepackage{graphpap}
\usepackage{subfigure}
\usepackage{enumerate}
\usepackage[pagebackref,hypertexnames=false, colorlinks, citecolor=red, linkcolor=red]{hyperref} 
\usepackage[backrefs]{amsrefs}
\usepackage{amsmath,amstext,amsthm,amssymb,amsxtra}

\input txdtools

\begin{document}

\makeatletter
\newcommand*{\defeq}{\mathrel{\rlap{%
                     \raisebox{0.3ex}{$\m@th\cdot$}}%
                     \raisebox{-0.3ex}{$\m@th\cdot$}}%
                     =}
										
\newcommand*{\eqdef}{=
										 \mathrel{\rlap{%
                     \raisebox{0.3ex}{$\m@th\cdot$}}%
                     \raisebox{-0.3ex}{$\m@th\cdot$}}%
										}
\makeatother

\newcommand{\unit}{1\!\!1}

\newcommand*\rfrac[2]{{}^{#1}\!/_{#2}}

\newcommand{\norm}[1]{\ensuremath{\left\|#1\right\|}}
\newcommand{\abs}[1]{\ensuremath{\left\vert#1\right\vert}}
\newcommand{\ip}[2]{\ensuremath{\left\langle#1,#2\right\rangle}}

\newcommand{\pbar}{\ensuremath{\bar{\partial}}}
\newcommand{\db}{\overline\partial}
\newcommand{\D}{\mathbb{D}}
\newcommand{\B}{\mathbb{B}}
\newcommand{\Sp}{\mathbb{S}}
\newcommand{\T}{\mathbb{T}}
\newcommand{\R}{\mathbb{R}}
\newcommand{\Z}{\mathbb{Z}}
\newcommand{\C}{\mathbb{C}}
\newcommand{\N}{\mathbb{N}}

\newcommand{\m}[1]{\mathcal{#1}}
\newcommand{\eps}{\epsilon}
\newcommand{\he}[1]{h_{#1}^{\eps}}
\newcommand{\hn}[1]{h_{#1}^{\eta}}
\newcommand{\avg}[1]{\langle #1 \rangle}

\newcommand{\mcd}{\mathcal{D}}
\newcommand{\mcr}{\mathcal{R}}
\newcommand{\mch}{\mathcal{H}}

\newcommand{\ep}{\epsilon}
\newcommand{\lb}{\lambda}

\newcommand{\La}{\langle }
\newcommand{\Ra}{\rangle }

\newcommand{\bmo}{\textnormal{BMO}}
\newcommand{\wbmod}{BMO_{\mcd}(w)}
\newcommand{\wbmosd}{BMO^2_{\mcd}(w)}
\newcommand{\nbmod}{BMO_{\mcd}(\nu)}
\newcommand{\nbmosd}{BMO^2_{\mcd}(\nu)}
\newcommand{\avgb}{\left<b\right>}
\newcommand{\h}[1]{h_{#1}^{0}}
\newcommand{\nch}[1]{h_{#1}^{1}}
\newcommand{\fd}{I_{\alpha}^{\m{D}}}
\newcommand{\BMO}{\textnormal{BMO}}

\numberwithin{equation}{section}

\newtheorem{thm}{Theorem}[section]
\newtheorem{lm}[thm]{Lemma}
\newtheorem{cor}[thm]{Corollary}
\newtheorem{conj}[thm]{Conjecture}
\newtheorem{prob}[thm]{Problem}
\newtheorem{prop}[thm]{Proposition}
\newtheorem*{prop*}{Proposition}

\theoremstyle{remark}
\newtheorem{rem}[thm]{Remark}
\newtheorem*{rem*}{Remark}

\title[Commutators with Fractional Integral Operators]
{Two-Weight Inequalities for Commutators with Fractional Integral Operators}

\author{Irina Holmes}
\address{Irina Holmes, School of Mathematics\\ Georgia Institute of Technology\\ 686 Cherry Street\\ Atlanta, GA USA 30332-0160}
\email{irina.holmes@math.gatech.edu}

\author{Robert Rahm}
\address{Robert Rahm, School of Mathematics\\ Washington University in St. Louis\\
  One Brookings Drive\\ St. Louis, MO USA 63130}
\email{rahm@wustl.edu}
\thanks{}

\author{Scott Spencer}
\address{Scott Spencer, School of Mathematics\\ Georgia Institute of Technology\\ 686 Cherry Street\\ Atlanta, GA USA 30332-0160}
\email{spencer@math.gatech.edu}
\thanks{}


\begin{abstract} 
In this paper we investigate weighted norm inequalities for the commutator of a fractional integral operator 
and multiplication by a function. In particular, we show that,
for $\mu,\lambda\in A_{p,q}$ and $\alpha/n+1/q=1/p$, the norm 
$\norm{[b,I_\alpha]:L^p(\mu^p)\to L^q(\lambda^q)}$ is equivalent to the norm of $b$ in the weighted BMO space 
$BMO(\nu)$, where $\nu=\mu\lambda^{-1}$. This work extends some of the results on this topic existing in the literature,
and continues a line of investigation which was initiated 
by Bloom in 1985 and was recently developed further by the first author, Lacey, and Wick.

\end{abstract}

\maketitle
\setcounter{tocdepth}{1}
\tableofcontents

\section{Introduction and Statement of Main Results}

Recall the classical fractional integral operator, or Riesz potential, on $\mathbb{R}^n$:
let $0 < \alpha < n$ be fixed and,\ for a Schwartz function $f$ define the 
fractional integral operator (or Riesz potential) $I_{\alpha}$ by
	$$I_{\alpha}f(x) 
	:= \int_{\R^n} \frac{f(y)}{|x-y|^{n-\alpha}} \,dy.$$
These operators have been studied since 1949, when they were introduced by Marcel Riesz, and have since found many applications in analysis -- such as Sobolev embedding theorems and PDEs.
Also recall the Calder{\'o}n-Zygmund operators:
	$$Tf(x):=\int_{\mathbb{R}^n} K(x,y)f(y)\,dy,\quad x\notin\textnormal{supp} f,$$	
where the kernel satisfies the standard size and smoothness estimates:
 \begin{gather*}
 \left\vert K(x,y)\right\vert  \leq  \frac{C}{\left\vert x-y\right\vert^n}, \\
 \left\vert K(x+h,y)-K(x,y)\right\vert +\left\vert K(x,y+h)-K(x,y)\right\vert  \leq  C\frac{\left\vert h\right\vert^{\delta}}{\left\vert x-y\right\vert^{n+\delta}},
 \end{gather*}
 for all $\left\vert x-y\right\vert>2\left\vert h\right\vert>0$ and a fixed $\delta\in (0,1]$.

To contrast the two, note for example that fractional integral operators are positive, which in many cases 
makes them easier to work with (as one example of this, it is almost trivial
to dominate the fractional integral operators by sparse operators, though
this isn't important to us in the present setting). 
On the other hand, the fractional integral
operators do not commute with dilations and therefore can never boundedly map
$L^p(dx)$ to itself. Additionally, the kernel of the fractional integral
operator does not satisfy the standard estimates above. Therefore, the theory
of fractional integral operators is not just a subset of the theory of
Calder\'on--Zygmund operators. Because of this, results which are known
for Calder{\'o}n-Zygmund operators also need to be proved for the fractional
integral operators. 
	
In this paper we will characterize the triples
$(b,\mu,\lambda)$, where $b$ is a function and $\mu$ and $\lambda$
are $A_{p,q}$ weights (to be defined shortly),
such that the commutator $[b,I_{\alpha}]$ is bounded from 
$L^p(\mu^p)$ to $L^q(\lambda^q)$.  Commutators with Riesz potentials were first studied in \cite{Chan1982}.

Our characterization will be in terms of
the norm of $b$ in a certain weighted $\bmo$ space, built
from the weights $\mu$ and $\lambda$. This is an adaptation to the fractional integral setting of a viewpoint
introduced by Bloom \cite{Blo1985} in 1985, and recently investigated
by the first author, Lacey and Wick in \cites{HolLacWic2015a,
HolLacWic2015b}. 
Specifically, Bloom characterized $\norm{[b, H] : L^p(\mu) \rightarrow L^p(\lb)}$, where $H$ is the Hilbert transform and $\mu, \lb$ are $A_p$ weights,
in terms of $\|b\|_{BMO(\nu)}$, where $BMO(\nu)$ is the weighted BMO space associated with the weight $\nu \defeq \mu^{1/p}\lb^{-1/p}$. Recall that the 
Hilbert transform is the one-dimensional prototype for Calder{\'o}n-Zygmund operators, a role played by the Riesz transforms in $\mathbb{R}^n$. 

A modern 
dyadic proof of Bloom's result was recently given in \cite{HolLacWic2015a}, and the techniques developed were then used to extend the result to
all Calder{\'o}n-Zygmund operators in \cite{HolLacWic2015b}. In particular, it was proved that 
	\begin{equation} \label{E:UB_HLW}
	\left\| [b, T] : L^p(\mu) \rightarrow L^p(\lb) \right\| \leq c \|b\|_{BMO(\nu)},
	\end{equation}
for all $A_p$ weights $\mu, \lb$, and all Calder{\'o}n-Zygmund operators $T$ on $\mathbb{R}^n$, for some constant $c$ depending on $n$, $T$, $\mu$, $\lb$ and $p$. Specializing to the Riesz transforms, a lower bound was also proved.
The center of the proof of \eqref{E:UB_HLW} was the Hyt{\"o}nen Representation Theorem, which allows one to recover $T$ from averaging over some dyadic operators, called dyadic shifts. Then the upper bound reduced to these dyadic operators. 

We take a similar approach in this paper, where the role of the dyadic shifts will be played by the dyadic version of the fractional integral operator $I_{\alpha}$, given by:
\begin{align}\label{D:dfo}
I_{\alpha}^{\m{D}}f:= 
  \sum_{Q\in\m{D}}\abs{Q}^{\alpha/n}\La f\Ra_Q \unit_Q.
\end{align}
Our main result is:

\begin{thm} \label{T:mainThm}
Suppose that $\alpha/n+1/q=1/p$ and $\mu, \lb \in A_{p,q}$. Let $\nu \defeq \mu\lb^{-1}$. 
Then:
\begin{align*}
\norm{[b,I_{\alpha}]:L^p(\mu^p)\to L^q(\lb^q)}
\simeq 
\norm{b}_{\bmo(\nu)}.
\end{align*}
\end{thm}

It is important to observe that we require that each weight
belong to a certain $A_{p,q}$ class and this will imply that
$\mu\lambda^{-1}$ is an $A_2$ weight and in particular, 
an $A_\infty$ weight. Standard properties of these weight
classes will be used throughout the paper, with out tracking
dependencies on the particular weight characteristics. The 
liberal use of these properties indicates the subtleties 
involved in the general two--weight setting. For an excellent 
account of this and other topics related to fractional integral 
operators, see \cite{Cru2015}.

The paper is organized as follows. In Section 2, we will give the
requisite background material and definitions. Note, however,  that most of the material not relating strictly to fractional integral operators 
(such as the Haar system, $A_p$ weights, and weighted BMO) is standard and
was also needed in \cite{HolLacWic2015b} where it is discussed in more detail.
In Section 3 we will briefly discuss how the fractional integral
operator can be recovered as an average of dyadic operators.
In Section 4 we 
will prove $\norm{[b,I_{\alpha}]:L^p(\mu^p)\to L^q(\lb^q)}
\lesssim \norm{b}_{\bmo(\nu)}$ and in Section 5, we will 
prove the reverse inequality: $\norm{b}_{\bmo(\nu)}
\lesssim \norm{[b,I_{\alpha}]:L^p(\mu^p)\to L^q(\lb^q)}$.

\section{Background and Notation} 

\subsection{The Haar System}
Let $\m{D}$ be a dyadic grid on $\mathbb{R}^n$ and let $Q\in\m{D}$.
For every $\eps\in\{0,1\}^{n}$, let $\he{Q}$ be the usual Haar function
defined on $Q$. For convenience, we write $\eps=1$ if $\eps=(1,1,\ldots,1)$. Note that, in this case, $\int h_{Q}^1=1$.
Otherwise, if $\eps\neq 1$, then $\int\he{Q}=0$. Moreover, recall that 
$\{\he{Q}\}_{Q\in\m{D},\eps\neq 1}$ forms an orthonormal basis
for $L^2(\mathbb{R}^n)$. For a function $f$, a cube $Q\in\m{D}$ and $\ep\neq 1$, we denote 
$$\widehat{f}(Q,\ep) \defeq \La f, h_Q^{\ep}\Ra,$$ 
where $\La\cdot,\cdot\Ra$ is the usual inner product in $L^2(\mathbb{R}^n)$.

\subsection{$A_p$ Classes and Weighted BMO} Let $w$ be a weight on $\mathbb{R}^n$, that is, a locally integrable, almost everywhere positive function. 
For a subset $Q\subset\mathbb{R}^n$ we denote 
$$ w(Q) \defeq \int_Q w\,dx \:\:\text{ and }\:\: \La w\Ra_Q \defeq \frac{w(Q)}{|Q|}.$$
Given $1<p<\infty$,  a weight $w$ is said to belong to the Muckenhoupt $A_p$ class  provided that:
	$$[w]_{A_p} \defeq \sup_{Q} \La w\Ra_Q \La w^{1-p'}\Ra_Q^{p-1} < \infty,$$
where $p'$ denotes the H\"{o}lder conjugate of $p$, and the supremum is over all cubes $Q \subset \R^n$. Moreover, $w \in A_p$ if and only if $w^{1-p'} \in A_{p'}$ and, in this case, $[w^{1-p'}]_{A_{p'}} = 
[w]_{A_p}^{p'-1}$. Furthermore, if $1<p<q<\infty$, then $A_p \subset A_q$, with $[w]_{A_q} \leq [w]_{A_p}$ for all $w \in A_p$. 

For a dyadic lattice $\m{D}$, recall the dyadic square function:
\begin{align*}
(S_{\m{D}}f)^2
=\sum_{P\in\m{D},\eps\neq 1}\abs{\widehat{f}(Q,\ep)}^2 \frac{\unit_Q}{|Q|}.
\end{align*}
Another property of $A_p$ weights which will be useful for us is the following well--known weighted 
Littlewood--Paley Theorem:

\begin{thm}\label{T:squareFunctBounds}
Let $w\in A_p$, then:
\begin{align*}
 \norm{S_{\m{D}}:L^p(w)\to L^p(w)}\simeq c(n, p, [w]_{A_p}).
\end{align*}
\end{thm}

For a weight $w$ on $\R^n$, the weighted BMO space $\BMO(w)$ is defined to be the space of all locally 
integrable functions $b$ that satisfy:
	\begin{equation} \label{E:wBMONorm1}
	\|b\|_{\BMO(w)} \defeq \sup_{Q} \frac{1}{w(Q)} \int_Q |b - \left<b\right>_Q|\,dx  < \infty,
	\end{equation}
where the supremum is over all cubes $Q$ in $\R^n$. 
For a general weight, the definition of the $\BMO$ norm is highly dependent on its $ L ^{1}$ average.  
But, if the weight is $ A_ \infty $, one is free to replace the $ L ^{1}$-norm by larger averages.  
Namely, defining 
	\begin{equation} \label{E:wBMONormq}
	\|b\|_{\BMO^q(w)} \defeq \sup_{Q} \left( \frac{1}{w(Q)} 
	\int_Q |b - \left<b\right>_Q|^q\,dw' \right)^{\frac{1}{q}},
	\end{equation}
there holds
	\begin{equation} \label{E:wBMO-qEquiv}
	\|b\|_{\BMO(w)} \leq \|b\|_{\BMO^q(w)} \leq C(n, p, [w]_{A_ {\infty}} ) \|b\|_{\BMO(w)}.
	\end{equation}
The proof is similar to the proof in the unweighted case. In particular,  
the first inequality is a straightforward application of H\"older's inequality and
the second inequality follows from a suitable John--Nirenberg property 
(which requires a suitable Calder\'{o}n--Zygmund decomposition). The details are 
in \cite{MucWhe1975}.
	
For a dyadic grid $\mcd$ on $\R^n$, we define the dyadic versions of the norms 
above by taking supremum over $Q \in \mcd$ instead of over all cubes $Q$ in 
$\R^n$, and denote these spaces by $\wbmod$ and $\BMO^q_{\mcd}(w)$. 
Clearly $\BMO(w) \subset \wbmod$ for any choice of $\mcd$, and the equivalence in
\eqref{E:wBMO-qEquiv} also holds for the dyadic versions of these spaces. 
	
A fact which will be crucial to our proof is the following:
\begin{lm}\label{l:wA2} If $ w \in A_2$, there holds 
	\begin{equation} \label{E:H1BMODual1}
	\left|\La b, \Phi\Ra\right| \lesssim   \|b\|_{\wbmosd} \|S_{\mcd}\Phi\|_{L^1(w)}.
	\end{equation}
\end{lm}
This comes from a duality relationship between dyadic weighted BMO spaces and dyadic weighted Hardy spaces. For a more detailed discussion and a proof of this fact, see Section 2.6 of \cite{HolLacWic2015b}. We remark here that Lemma \ref{l:wA2} was also fundamental for the proof of the upper bound \eqref{E:UB_HLW} in \cite{HolLacWic2015b}, essentially for the following reason: if $\mu, \lb$ are $A_p$ weights, then $\nu := \mu^{1/p}\lb^{-1/p}$ is an $A_2$ weight. Thus the duality statement above applied to $\nu$ eventually yields, through H{\"o}lder's inequality, some bounds in terms of $L^p(\mu)$ and $L^{p'}(\lb)$ norms. This is also the strategy we will adapt accordingly to the fractional integral case, which makes use of $A_{p,q}$ classes instead. We discuss these next.

\subsection{$A_{p,q}$ Classes} 

Throughout this section, $\alpha,n,p,q$ are fixed and satisfy 
$1/p - 1/q = \alpha/n$. We recall first the fractional maximal operator,
	$$M_{\alpha}f \defeq \sup_Q |Q|^{\alpha/n} \La |f| \Ra_Q \unit_Q,$$
with the supremum being over all cubes $Q$. This was first introduced in  \cite{MuckWheedenFrac}, where it was used to prove weighted inequalities for $I_{\alpha}$, a result analogous to the classic result \cite{CoifFeff} of Coifman and Fefferman, relating the Hardy-Littlewood maximal operator and singular integrals. We will be working with the dyadic version of this operator, $M_{\alpha}^{\m{D}}$, defined for a dyadic grid $\m{D}$ just as above, but only taking supremum over $Q\in\m{D}$.

Also in \cite{MuckWheedenFrac} was introduced a generalization of $A_p$ classes for the fractional integral setting: we say that a weight $w$ belongs to the $A_{p,q}$ class provided that 
	$$[w]_{A_{p,q}} \defeq \sup_Q \La w^q\Ra_{Q} \La w^{-p'}\Ra_Q^{\rfrac{q}{p'}} < \infty.$$
See \cites{Roch1993,RahSpe2015,CruMoe12013b,CruMoe2013a,Cru2015} for other
generalizations.

We will use the following
important result concerning $A_{p,q}$ weights due to, for example, Sawyer and Muckenhoupt and Wheeden
\cites{Saw1982,Saw1988,MuckWheedenFrac}:
\begin{thm}\label{T:maxIntBounds}
Let $w$ be a weight. Then the following are equivalent:
\begin{itemize}
 \item [(i)] $w\in A_{p,q}$;\\
 \item [(ii)] $\norm{M_{\alpha}^{\m{D}}:L^p(w^p)\to L^q(w^q)}
  \simeq C(n, \alpha, p, [w]_{A_{p,q}})$;\\
 \item [(iii)] $\norm{I_{\alpha}^{\m{D}}:L^p(w^p)\to L^q(w^q)}
  \simeq C(n, \alpha, p, [w]_{A_{p,q}})$.
\end{itemize}
\end{thm}

We now make two observations about $A_{p,q}$ weights which will be particularly useful to us. First, we note that:
	\begin{equation} \label{E:Mucks}
	\text{If } w \in A_{p,q} \text{, then: } w^p \in A_p,\:\: w^{-p'} \in A_{p'},\:\: w^q \in A_q\text{, and }\: w^{-q'} \in A_{q'},
	\end{equation}
where all weights above have Muckenhoupt characteristics bounded by powers of $[w]_{A_{p,q}}$.	
To see that $w^p \in A_p$, first notice $w \in A_{p,q}$ if and only if $w^q \in A_{q_0}$, with $[w^q]_{A_{q_0}} = [w]_{A_{p,q}}$, where
	$$q_0 \defeq 1 + q/p'= q(1 - \alpha/n).$$
Since the $A_p$ classes are increasing and $q_0 < q$, we have that $w^q \in A_q$. In turn, this gives that $w^{-q'} = (w^q)^{1-q'} \in A_{q'}$.  The other two statements in \eqref{E:Mucks}
follow in a similar fashion from the fact that $w \in A_{p,q}$ if and only if $w^{-1} \in A_{q', p'}$.
	
Second, suppose that $\mu, \lb \in A_{p,q}$ and let $\nu \defeq \mu\lb^{-1}$. Since $\mu^p, \lb^p \in A_p$, H\"older's inequality implies $\nu \in A_2$ (with $[\nu]_{A_2}^p \leq [\mu^p]_{A_p}[\lambda^p]_{A_p}$), a fact which will be used in proving the upper bound. Moreover, we claim that for any cube $Q$:
	\begin{equation} \label{E:wtEst}
	\mu^p(Q)^{1/p} \lb^{-q'}(Q)^{1/q'} \lesssim \nu(Q) |Q|^{\alpha/n},
	\end{equation}
a fact which will be useful in proving the lower bound. To see this, note first that 
	\begin{equation*}
	\La \mu^p\Ra_Q^{1/p} \La \mu^{-p'}\Ra_Q^{1/p'} \lesssim 1 \:\:\:\text{ and }\:\:\:
	\La\lb^{-q'}\Ra_Q^{1/q'} \La \lb^q\Ra_Q^{1/q} \lesssim 1,
	\end{equation*}
which simply come from $\mu^p \in A_p$ and $\lb^q \in A_q$. Since $p' > q'$, H\"older implies
	\begin{align*}\left( \frac{1}{|Q|} \int_Q \mu^{-q'}\,dx\right)^{1/q'}  &\leq \left( \frac{1}{|Q|} \left(\int_Q \mu^{-p'}\,dx\right)^{q'/p'} \left(\int_Q\,dx\right)^{1 - q'/p'}\right)^{1/q'}\\
	 &= \left( \frac{1}{|Q|} \int_Q \mu^{-p'}\,dx\right)^{1/p'},
	\end{align*}
and hence $\La \mu^{-q'}\Ra_Q^{1/q'} \leq \La \mu^{-p'}\Ra_Q^{1/p'}.$
Combining these estimates gives:
	$$ \La \mu^p\Ra_Q^{1/p} \La \lb^{-q'}\Ra_Q^{1/q'} \lesssim
		\frac{1}{\La \mu^{-p'}\Ra_Q^{1/p'}} \frac{1}{\La \lb^q\Ra_Q^{1/q}} \lesssim
		\frac{1}{ \La \mu^{-q'}\Ra_Q^{1/q'} \La \lb^q\Ra_Q^{1/q} } \leq \frac{1}{\La\nu^{-1}\Ra_Q} \leq \La \nu\Ra_Q.$$
The last two inequalities are more application of H{\"o}lder's inequality and the fact that $\nu^{-1} = \mu^{-1}\lb$.  This proves \eqref{E:wtEst}.


\section{Averaging Over Dyadic Fractional Integral Operators}
In this section, we show that $I_\alpha$ can be recovered
from \eqref{D:dfo}
by averaging over dyadic lattices. The proof here is modified
(and abridged) 
from the proof in \cite{PetTreVol2002}, but it is possible 
to modify any of the proofs in, for example, 
\cites{Pet2000,Hyt2008,Lac2009}. For the sake of clarity, we
only give the proof for the one--dimensional case.

Given an interval $[a,b)$ (it is not too important that the 
interval be closed on the left and open on the right) of
length $r$, we can create a dyadic lattice, $\m{D}_{a,r}$
in a standard way. In particular, $\m{D}_{a,r}$ is the
dyadic lattice on $\mathbb{R}$ with intervals of length
$r2^{-k}$, $k\in\mathbb{Z}$, and the point $a$ is not in
the interior of any of the intervals in $\m{D}_{a,r}$.
For example, $\m{D}_{0,1}$ is the standard dyadic lattice
on $\mathbb{R}$. For a given lattice $\m{D}_{a,r}$, we
let $\m{D}_{a,r}^{k}$ denote the intervals in 
$\m{D}_{a,r}$ with length $r2^{-k}$. In this section we 
slightly abuse notation and let $h_{I}^{1}=\abs{I}^{-1/2}\unit_{I}$.

Define:
\begin{align*}
\mathbb{P}_{(a,r)}^{0}f(x)
:=\sum_{I\in\m{D}_{a,r}^{0}}\abs{I}^{\alpha}\ip{f}{\nch{I}}\nch{I}(x).
\end{align*}
With $r$ and $x$ fixed, we can parameterize the dyadic
grids by the set $(-r,0]$ and we can give this set 
the probability measure $da/r$. 
For a fixed $x\in\mathbb{R}$, 
we want to compute:
\begin{align*}
\mathbb{E}(\mathbb{P}_{(a,r)}^{0}f(x))
=\int_{-r}^{0}\mathbb{P}_{(a,r)}^{0}f(x)\frac{da}{r}.
\end{align*}
Let $\tau_{t}f(x):=f(x+t)$ be the translation
operator and note that $\mathbb{P}_{a-t}\tau_t
=\tau_t\mathbb{P}_a$. From this it easily follows
that $\mathbb{E}\mathbb{P}_{(a,r)}^{0}\tau_t
=\tau_t\mathbb{P}_{(a,r)}^{0}$. That is, 
$\mathbb{E}\mathbb{P}_{(a,r)}^{0}$ is given
by convolution. Let:
\begin{align*}
\mathbb{E}\mathbb{P}_{(a,r)}^{0}f(x)=F_{0,r}\ast f(x).
\end{align*}
We want to compute $F_{0,r}$. First, note that 
$\mathbb{P}_{a,r}^{0}$ is convolution with the function
$\frac{r^{\alpha}}{r}\unit_{[-r/2,r/2]}$. Therefore, we have:
\begin{align*}
F_{0,r}\ast f(x)
&=\mathbb{E}\mathbb{P}_{(a,r)}^{0}f(x)
\\&=\mathbb{E}\mathbb{P}_{(a/2,r)}^{0}f(x)
\int_{x-r/2}^{x+r/2}\int_{\mathbb{R}}
  f(s)\frac{r^\alpha}{r}\unit_{-r/2,r/2}(t-s)ds\frac{dt}{r}.
\end{align*}
Using Fubini, we see that:
\begin{align*}
F_{0,r}(x)
=\int_{x-r/2}^{x+r/2}\frac{r^\alpha}{r}\unit_{[-r/2,r/2]}(t)\frac{dt}{r}
=\frac{r^\alpha}{r}\unit_{[-r/2,r/2]}(x)\left(1-\abs{\frac{x}{r}}\right)
=\frac{r^\alpha}{r}F_{0,1}(x/r).
\end{align*}
Now, fix an $r\in [1,2)$ and define:
\begin{align*}
F_r=\sum_{n\in\mathbb{Z}}F_{0,2^{n}r}.
\end{align*}
The grids $\m{D}_{a,r}^{k},k\in\mathbb{Z}$ 
can be unioned to form a dyadic lattice (here $a$ is fixed).
Call $r$ the calibre of the dyadic lattice. Convolution 
with $F_r$ is averaging over all the dyadic lattices
$\m{D}_{a,r}$ with fixed calibre $r$. That is:
\begin{align*}
F_r\ast f = \mathbb{E}\mathbb{P}_{\m{D}_{a,r}}f.
\end{align*}
Finally, we need to average over $r\in[1,2)$. Set
$F(x):=\int_{1}^{2}F_r(x)\frac{dr}{r}$. Now, 
we want to compute $F(x)$. There holds:
\begin{align*}
F(x)
&=\int_{1}^{2}F_r(x)\frac{dr}{r}
\\&=\int_{1}^{2}\sum_{n\in\mathbb{Z}}F_{0,2^{n}r}(x)\frac{dr}{r}
\\&=\int_{0}^{\infty}F_{0,\rho}(x)\frac{d\rho}{\rho}
\\&=\int_{0}^{\infty}F_{0,1}(\frac{x}{\rho})\frac{\rho^\alpha}{\rho^2}d\rho
\\&=\int_{0}^{\infty}\unit_{-1/2,1/2}(\frac{x}{\rho})
  (1-\abs{\frac{x}{\rho}})
  \frac{\rho^\alpha}{\rho^2}dr.
\end{align*}
Now, if $x>0$, making the change of variable $t=x/\rho$, we see:
\begin{align*}
F(x)=\frac{x^\alpha}{x}\int_{0}^{\infty}F_{0,1}(y)\frac{dy}{y^\alpha}
=c_\alpha\frac{1}{x^{1-\alpha}}.
\end{align*}
Doing a similar computation for when $x<0$, we see that
$F(x)=c_\alpha \frac{1}{\abs{x}^{1-\alpha}}$.


\section{Upper Bound}
The decomposition in Section 3 
means that the upper bound in Theorem \ref{T:mainThm} 
follows from the following, where the implied 
constants are independent of the dyadic lattice:
\begin{lm} \label{T:mainLem}
Suppose that $\alpha/n+1/q=1/p$ and $\mu, \lb \in A_{p,q}$. Let $\nu \defeq \mu\lb^{-1}$. Then:
\begin{align*}
\norm{[b,\fd]:L^p(\mu^p)\to L^q(\lb^q)}
\lesssim
\norm{b}_{\bmo(\nu)}.
\end{align*}
\end{lm}

\begin{proof}
We show that $[b,I_{\alpha}^{\m{D}}]$ can be decomposed
as the sum of four operators which will be fairly easy to bound. First note that for $\eps\neq 1$, there holds:
\begin{align*}
I_{\alpha}^{\m{D}}h_Q^{\eps}
=\sum_{P\in\m{D}:P\subsetneq Q}\abs{P}^{\alpha/n}h_{Q}^{\eps}(P)
  \unit_{P}
=\bigg(\sum_{P\in\m{D}:P\subsetneq Q}\abs{P}^{\alpha/n}\unit_{P}
  \bigg)\he{Q}
=c_{\alpha}\abs{Q}^{\alpha/n}\he{Q}.
\end{align*}
Similarly, 
\begin{align*}
I_{\alpha}^{\m{D}}\unit_{Q}
=(1+c_{\alpha})\abs{Q}^{\alpha/n}\unit_{Q}
  +\abs{Q}\sum_{R\in\m{D}:Q\subsetneq R}
    \abs{R}^{\alpha/n}\frac{\unit_{R}}{\abs{R}}.
\end{align*}
Using these computations:
\begin{align*}
I^{\mcd}_{\alpha} (h_P^{\ep} h_Q^{\eta}) = \left\{
		\begin{array}{ll}
		c_{\alpha} |P \cap Q|^{\frac{\alpha}{n}} h_P^{\ep} h_Q^{\eta} & \text{, if } P \neq Q \text{ or if } P=Q \text{ and } \ep\neq \eta;\\
		(1 + c_{\alpha}) |Q|^{\frac{\alpha}{n}} \frac{\unit_Q}{|Q|} + \sum_{R \supsetneq Q} |R|^{\frac{\alpha}{n}} \frac{\unit_R}{|R|} & \text{, if } P = Q \text{ and } \ep=\eta.
		\end{array}
	\right.
\end{align*}
Thus:
\begin{align*}
	[h_P^{\ep}, I_{\alpha}^{\mcd}] h^{\eta}_Q = \left\{
		\begin{array}{ll}
		c_{\alpha} h_Q^{\eta}(P)h_P^{\ep} \left( |Q|^{\frac{\alpha}{n}} - |P|^{\frac{\alpha}{n}} \right) & \text{, if } P \subsetneq Q;\\
		-|Q|^{\frac{\alpha}{n}} \frac{\unit_Q}{|Q|} - \sum_{R \supsetneq Q} |R|^{\frac{\alpha}{n}} \frac{\unit_R}{|R|} & \text{, if } P = Q \text{ and } \ep=\eta; \\
		0 & \text{, if } Q \subsetneq P \text{, or if } Q = P \text{ and } \ep\neq\eta. 
		\end{array}
	\right.
\end{align*}
Expressing $b$ and $f$ in terms of their Haar coefficients, we obtain that
	$$
	[b, I^{\mcd}_{\alpha}] f = \sum_{P, Q \in \mcd} 
	\sum_{\ep,\eta\neq 1} \widehat{b}(P,\ep) \widehat{f}(Q,\eta) 
	[h_P^{\ep}, I_{\alpha}^{\mcd}] h^{\eta}_Q.
	$$
Using this, there holds
\begin{align}\label{E:Decomp}
[b, I^{\mcd}_{\alpha}]f = 
  c_{\alpha} T_1 f - c_{\alpha} \Pi_{b,\alpha}^{(0,1,0)}f 
  - \Pi_{b,\alpha}^{(0,0,1)}f - T_2f,
\end{align}
where:
\begin{align*}
\Pi_{b,\alpha}^{(0,1,0)} f &\defeq  \sum_{Q\in\mcd,\ep\neq 1} \widehat{b}(Q,\ep) \La f\Ra_Q |Q|^{\frac{\alpha}{n}} h_Q^{\ep};\\
\Pi_{b,\alpha}^{(0,0,1)} f & \defeq \sum_{Q\in\mcd,\ep\neq 1} \widehat{b}(Q,\ep) \widehat{f}(Q,\ep) |Q|^{\frac{\alpha}{n}} \frac{\unit_Q}{|Q|};\\
T_1 f & \defeq \sum_{P\in\mcd,\ep\neq 1} \widehat{b}(P,\ep) \left( \sum_{Q \supsetneq P, \eta\neq 1} \widehat{f}(Q,\eta) h_Q^{\eta}(P) |Q|^{\frac{\alpha}{n}} \right) h_P^{\ep};\\
T_2 f & \defeq \sum_{P\in\mcd,\ep\neq 1} \widehat{b}(P,\ep) \widehat{f}(P,\ep) \left( \sum_{Q \supsetneq P} |Q|^{\frac{\alpha}{n}} \frac{\unit_Q}{|Q|} \right).
\end{align*}

We will show that all of these operators are bounded 
$L^p(\mu^p) \rightarrow L^q(\lb^q)$. Below, all implied 
constants are allowed to depend on $n, \alpha, p, 
[\mu]_{A_{p,q}}$, and $[\lb]_{A_{p,q}}$. Also all inner products
below are taken with respect to $dx$ and therefore it
is enough to show:
\begin{align*}
\left \vert \ip{Tf}{g}\right\vert
\lesssim \norm{b}_{\bmo(\nu)}\norm{f}_{L^p(\mu^p)}
  \norm{g}_{L^{q'}(\lambda^{-q'})},
\end{align*}
for each of the four operators above (this is because the dual
of $L^q(\lambda^q)$ with respect to the unweighted inner product 
is $L^{q'}(\lambda^{-q'})$). The idea, which is
taken from \cites{HolLacWic2015a,HolLacWic2015b}, is to
write the bilinear form, $\ip{Tf}{g}$ as 
$\ip{b}{\Phi}$ and then show that $\norm{S_{\m{D}}\Phi}_{L^1(\nu)}$ 
is controlled by $\norm{f}_{L^p(\mu^p)} 
\norm{g}_{L^{q'}(\lambda^{-q'})}$; by the weighted $H^1-\bmo$ duality, 
this is enough to prove the claim. 

The estimates for the two paraproducts are almost identical, and 
we only give the proof for $\Pi_{b,\alpha}^{(0,1,0)}$. First
with
\begin{align*}
\Phi \defeq \sum_{Q\in\mcd,\ep\neq 1} 
  \avg{f}_Q |Q|^{\frac{\alpha}{n}} \widehat{g}(Q,\ep) h_Q^{\ep},
\end{align*}
there holds:
\begin{align*}
\ip{\Pi_{b,\alpha}^{(0,1,0)}f}{g}=\ip{b}{\Phi}.
\end{align*}
Then:
\begin{align*}
(S_{\mcd}\Phi)^2 = 
  \sum_{Q\in\mcd,\ep\neq 1} 
  |\La f\Ra_Q|^2 |Q|^{\frac{2\alpha}{n}} 
  |\widehat{g}(Q,\ep)|^2 \frac{\unit_Q}{|Q|} 
  \leq (M_{\alpha}f)^2 (S_{\mcd}g)^2.
\end{align*}
Therefore,
\begin{align*}
\|S_{\mcd}\Phi\|_{L^1(\nu)} 
  \leq \|M_{\alpha}f\|_{L^q(\mu^q)} \|S_{\mcd}g\|_{L^{q'}(\lb^{-q'})} 
  \lesssim \|f\|_{L^p(\mu^p)} \|g\|_{L^{q'}(\lb^{-q'})},
\end{align*}
where the last inequality follows from Theorem \ref{T:maxIntBounds} 
for the fractional maximal function, and from 
Theorem \ref{T:squareFunctBounds} and the fact 
that $\lb^{-q'} \in A_{q'}$ for the dyadic square function. 
The proof for $\Pi_{b, \alpha}^{(0,0,1)}$ is very similar, and
we omit the details. 

Now let us look at $T_1$. As above, we have $\ip{T_1 f}{g}=\ip{b}{\Phi}$, with 
\begin{align*}
\Phi:=
\sum_{P\in\mcd, \ep\neq 1} 
  \widehat{g}(P,\ep) \left( \sum_{Q \supsetneq P, \eta\neq 1} 
  \widehat{f}(Q,\eta) h_Q^{\eta}(P) |Q|^{\frac{\alpha}{n}} \right) 
  h_P^{\ep},
\end{align*}
Then:
\begin{align*}
(S_{\mcd} \Phi)^2 
& \leq \sum_{P\in\mcd,\ep\neq 1} |\widehat{g}(P,\ep)|^2 
  \left( \sum_{Q\supsetneq P, \eta\neq 1} \La |f|\Ra_Q 
  |Q|^{\frac{\alpha}{n}} \right)^2 \frac{\unit_P}{|P|}
\\&\leq (I_{\alpha}^{\mcd}|f|)^2 (S_{\mcd}g)^2.
\end{align*}
From Theorem \ref{T:maxIntBounds} and Theorem \ref{T:squareFunctBounds}, 
it follows that
\begin{align*}
\|S_{\mcd}\Phi\|_{L^1(\nu)} 
\leq \|I_{\alpha}^{\mcd} |f| \|_{L^q(\mu^q)} 
  \|S_{\mcd}g\|_{L^{q'}(\lb^{-q'})} 
\lesssim \|f\|_{L^p(\mu^p)} \|g\|_{L^{q'}(\lb^{-q'})}.
\end{align*}
The estimates for $T_2$ are similar and we omit the details.
\end{proof}


\section{Lower Bound}
In this section, we prove the lower bound in Theorem
\ref{T:mainThm}, which follows immediately from the Lemma below. In particular, 
we will show the following:
\begin{lm}
For all cubes, $Q$:
\begin{align*}
\frac{1}{\nu(Q)}\int_{Q}\abs{b(x)-\avg{b}_{Q}}dx
\lesssim \norm{[b,I_\alpha]:L^p(\mu^p)\to L^q(\lambda^q)}.
\end{align*}
\end{lm}
\begin{proof}
The proof here follows along the lines of the proof in 
\cite{Chaf2014}. 
We first make some reductions. As with unweighted $\bmo$, we can
replace the $\avg{b}_{Q}$ with any constant. Indeed, there holds:
\begin{align*}
\frac{1}{\nu(Q)}\int_{Q}\abs{b(x)-\avg{b}_{Q}}dx
&\leq \frac{1}{\nu(Q)}\int_{Q}\abs{b(x)-C_Q}dx
  +\frac{\abs{Q}}{\nu(Q)}\abs{C_Q-\avg{b}_{Q}}
\\&\leq \frac{2}{\nu(Q)}\int_{Q}\abs{b(x)-C_Q}dx.
\end{align*}
Second, let $P$ be the cube with $l(P)=4l(Q)$, where $l(Q)$ is 
the side length of $Q$, and
with the same ``bottom left corner'' as $Q$. By the doubling property of $A_\infty$ weights, there holds
$\nu(P)\simeq\nu(Q)$, and therefore it is enough to prove:
\begin{align*}
\frac{1}{\nu(P)}\int_{Q}\abs{b(x)-C_Q}dx
\lesssim \norm{[b,I_\alpha]:L^p(\mu^p)\to L^q(\lambda^q)}.
\end{align*}
Finally, let $P_R$ be the ``upper right half'' of $P$. Below, we will
use $C_Q=\avg{b}_{P_R}$. 

Now, for $x\in Q$ and $y\in P_R$ there holds:
\begin{align*}
\frac{\abs{x-y}}{2\sqrt{n}\abs{P}^{1/n}}
\geq \frac{\sqrt{n}\abs{Q}^{1/n}}{2\sqrt{n}\abs{P}^{1/n}}
=\frac{1}{8}
\hspace{.2in}
\textnormal{ and }
\hspace{.2in}
\frac{\abs{x-y}}{2\sqrt{n}\abs{P}^{1/n}}
\leq \frac{\sqrt{n}\abs{P}^{1/n}}{2\sqrt{n}\abs{P}^{1/n}}
\leq \frac{1}{2}.
\end{align*}
The point is that there is a function, $K(x)$, that is smooth
on $[-1,1]^{n}$, has a smooth periodic extension to $\mathbb{R}^n$, and is
equal to $\abs{x}^{n-\alpha}$ for $1/8\leq\abs{x}\leq 1/2$.
Therefore, for $x\in Q$ and $y\in P_R$ there holds:
\begin{align*}
\frac{\abs{x-y}}{2\sqrt{n}\abs{J}}
=K\left(\frac{x-y}{2\sqrt{n}\abs{J}}\right).
\end{align*}
Important for us is the fact that $K$ has a Fourier 
expansion with summable coefficients. 

We are now ready to prove the main estimate. First, let
$\sigma(x)=\textnormal{sgn}(b(x)-\avg{b}_{P_R})$. Then:
\begin{align*}
\int_{Q}\abs{b(x)-\avg{b}_{P_R}}dx
&=\frac{1}{\abs{P_R}}\int_{\mathbb{R}}\int_{\mathbb{R}}
  (b(x)-b(y))\sigma(x)\unit_{Q}(x)\unit_{P_R}(y)dydx
\\&=\frac{1}{\abs{P_R}}\int_{\mathbb{R}}\int_{\mathbb{R}}
  \frac{b(x)-b(y)}{\left(\frac{\abs{x-y}}{2\sqrt{n}\abs{P}}\right)^{n-\alpha}}
  \left(\frac{\abs{x-y}}{2\sqrt{n}\abs{P}}\right)^{n-\alpha}
  \sigma(x)\unit_{Q}(x)\unit_{P_R}(y)dydx
\\&\simeq \abs{P}^{-\alpha/n}\int_{\mathbb{R}}\int_{\mathbb{R}}
  \frac{b(x)-b(y)}{\abs{x-y}^{n-\alpha}}
  K\left(\frac{x-y}{2\sqrt{n}\abs{P}}\right)
  \sigma(x)\unit_{Q}(x)\unit_{P_R}(y)dydx.
\end{align*}
Observe that the integral above is positive, so the ``$\simeq$''
is not a problem. Expanding $K$ in its Fourier series:
\begin{align*}
K\left(\frac{x-y}{2\sqrt{n}\abs{P}}\right)
=\sum_{k}a_ke^{ikx/2\sqrt{n}\abs{P}}e^{-iky/2\sqrt{n}\abs{P}},
\end{align*}
and inserting this into the integral, we continue:
\begin{align*}
\abs{P}^{-\alpha}\sum_{k}a_k\int_{Q}\int_{P_R}
  \frac{b(x)-b(y)}{\abs{x-y}^{n-\alpha}}
  \sigma(x)e^{ikx/c\abs{P}}e^{-iky/c\abs{P}}dydx
=\abs{P}^{-\alpha}\sum_{k}a_k\int_{\mathbb{R}}
  h_k(x)[b,I_\alpha]f_k(x)dx,
\end{align*}
where $h_k(x)=\sigma(x)e^{ikx/c\abs{P}}\unit_{P}(x)$ and
$f_k(y)=e^{-iky/c\abs{P}}\unit_{P_R}(y)$. We control the integral
by:
\begin{align*}
\int_{\mathbb{R}}h_k(x)[b,I_\alpha]f_k(x)dx
&\leq \norm{[b,I_\alpha]:L^p(\mu^p)\to L^q(\lambda^q)}
  \norm{f_k}_{L^p(\mu^p)}\norm{h_k}_{L^{q'}(\lambda^{-q'})}
\\&=\norm{[b,I_\alpha]:L^p(\mu^p)\to L^q(\lambda^q)}
  \mu^{p}(P_R)^{1/p}\lambda^{-q'}(P)^{1/q'}
\\&=\norm{[b,I_\alpha]:L^p(\mu^p)\to L^q(\lambda^q)}
  \mu^{p}(P)^{1/p}\lambda^{-q'}(P)^{1/q'}.
\end{align*}
By \eqref{E:wtEst}, this is dominated by:
\begin{align*}
\norm{[b,I_\alpha]:L^p(\mu^p)\to L^q(\lambda^q)}
  \abs{P}^{\alpha}\nu(P).
\end{align*}
This completes the proof.
\end{proof}

\noindent \textbf{Acknowledgements.} All three authors would like to thank 
Michael Lacey and Brett Wick for suggesting the problem, and for many useful 
conversations on this topic.  The second author completed most of his 
portion of the work while still a student at Georgia Tech.



\begin{bibdiv}
  
\begin{biblist}
\bib{Blo1985}{article}{
   author={Bloom, Steven},
   title={A commutator theorem and weighted BMO},
   journal={Trans. Amer. Math. Soc.},
   volume={292},
   date={1985},
   number={1},
   pages={103--122}
}


\bib{Chaf2014}{article}{
    author={Chaffee, Lucas},
    title={Characterizations of BMO through 
    commutators of bilinear singular integral
    operators},
    date={2014}
    eprint={http://arxiv.org/pdf/1410.4587v3.pdf}
}


\bib{Chan1982}{article}{
   author={Chanillo, S.},
   title={A note on commutators},
   journal={Indiana Univ. Math. J.},
   volume={31},
   date={1982},
   number={1},
   pages={7--16},
   issn={0022-2518}
}

\bib{CoifFeff}{article}{
	author={Coifman, },
	author={Fefferman, },
	title={Weighted norm inequalities for maximal functions and singular integrals},
	date={1974},
}

\bib{Cru2015}{article}{
    author={Cruz-Uribe, David},
    title={Two weight norm inequalities for 
    fractional integral operators and commutators},
    date={2015},
    eprint={http://arxiv.org/abs/1412.4157}
}

\bib{CruMoe2013a}{article}{
   author={Cruz-Uribe, David},
   author={Moen, Kabe},
   title={A fractional Muckenhoupt-Wheeden theorem and its consequences},
   journal={Integral Equations Operator Theory},
   volume={76},
   date={2013},
   number={3},
   pages={421--446}
}

\bib{CruMoe12013b}{article}{
   author={Cruz-Uribe, David},
   author={Moen, Kabe},
   title={One and two weight norm inequalities for Riesz potentials},
   journal={Illinois J. Math.},
   volume={57},
   date={2013},
   number={1},
   pages={295--323}
}

\bib{DraEtAl2005}{article}{
   author={Dragi{\v{c}}evi{\'c}, Oliver},
   author={Grafakos, Loukas},
   author={Pereyra, Mar{\'{\i}}a Cristina},
   author={Petermichl, Stefanie},
   title={Extrapolation and sharp norm estimates for classical operators on
   weighted Lebesgue spaces},
   journal={Publ. Mat.},
   volume={49},
   date={2005},
   number={1},
   pages={73--91}
}


\bib{GarMar2001}{article}{
   author={Garc{\'{\i}}a-Cuerva, J.},
   author={Martell, J. M.},
   title={Wavelet characterization of weighted spaces},
   journal={J. Geom. Anal.},
   volume={11},
   date={2001},
   number={2},
   pages={241--264}
}


\bib{GelShi1964}{book}{
   author={Gel{\cprime}fand, I. M.},
   author={Shilov, G. E.},
   title={Generalized functions. Vol. 1},
   note={Properties and operations;
   Translated from the Russian by Eugene Saletan},
   publisher={Academic Press [Harcourt Brace Jovanovich, Publishers], New
   York-London},
   date={1964 [1977]},
   pages={xviii+423}
}

\bib{Gra2014}{book}{
   author={Grafakos, Loukas},
   title={Classical Fourier analysis},
   series={Graduate Texts in Mathematics},
   volume={249},
   edition={3},
   publisher={Springer, New York},
   date={2014},
   pages={xviii+638},
   isbn={978-1-4939-1193-6}
}


\bib{HolLacWic2015a}{article}{
    author={Holmes, Irina},
    author={Lacey, Michael T.},
    author={Wick, Brett D.},
    title={Bloom's Inequality: Commutators in a Two-Weight Setting},
    date={2015},
    eprint={http://arxiv.org/abs/1505.07947}
}


\bib{HolLacWic2015b}{article}{
	author={Holmes, I.},
	author={Lacey, Michael T.},
	author={Wick, Brett D.},
  	title={Commutators in the Two-Weight Setting},
  	journal={arXiv preprint arXiv:1506.05747},
  	year={2015},
}


\bib{Hyt2008}{article}{
   author={Hyt{\"o}nen, Tuomas},
   title={On Petermichl's dyadic shift and the Hilbert transform},
   language={English, with English and French summaries},
   journal={C. R. Math. Acad. Sci. Paris},
   volume={346},
   date={2008},
   number={21-22},
   pages={1133--1136}
}


\bib{Hyt2012}{article}{
   author={Hyt{\"o}nen, Tuomas P.},
   title={The sharp weighted bound for general Calder\'on-Zygmund operators},
   journal={Ann. of Math. (2)},
   volume={175},
   date={2012},
   number={3},
   pages={1473--1506}
}


\bib{Isr2015}{article}{
    author={Isralowitz, Joshua},
    title={A Matrix Weighted $T1$ Theorem for Matrix
    Kernelled CZOs and a Matrix Weighted John--Nirinberg
    Theorem},
    date={2015},
    eprint={http://arxiv.org/abs/1508.02474}
}


\bib{Lac2007}{article}{
   author={Lacey, Michael T.},
   title={Commutators with Reisz potentials in one and several parameters},
   journal={Hokkaido Math. J.},
   volume={36},
   date={2007},
   number={1},
   pages={175--191}
}


\bib{Lac2009}{article}{
   author={Lacey, Michael},
   title={Haar shifts, commutators, and Hankel operators},
   journal={Rev. Un. Mat. Argentina},
   volume={50},
   date={2009},
   number={2},
   pages={1--13},
   issn={0041-6932},
   review={\MR{2656521 (2011f:47045)}},
}


\bib{LeeLinLin2009}{article}{
   author={Lee, Ming-Yi},
   author={Lin, Chin-Cheng},
   author={Lin, Ying-Chieh},
   title={A wavelet characterization for the dual of weighted Hardy spaces},
   journal={Proc. Amer. Math. Soc.},
   volume={137},
   date={2009},
   number={12},
   pages={4219--4225}
}


\bib{MuckWheedenFrac}{article}{
	author={Muckenhoupt, B.},
	author={Wheeden, R. L.},
	title={Weighted norm inequalities for fractional integrals},
	journal={Trans. Amer. Math. Soc.},
	volume={192},
	pages={261--274},
	date={1974},
}


\bib{MucWhe1975}{article}{
  author={Muckenhoupt, B.},
  author={Wheeden, R. L.},
  title={Weighted bounded mean oscillation and the Hilbert transform},
  journal={Studia Math.},
  volume={54},
  date={1975/76},
  number={3},
  pages={221--237}
}


\bib{OldSpa1987}{article}{
   author={Oldham, K. B.},
   author={Spanier, J.},
   title={Fractional calculus and its applications},
   language={English, with Romanian summary},
   journal={Bul. Inst. Politehn. Ia\c si Sec\c t. I},
   volume={24(28)},
   date={1978},
   number={3-4},
   pages={29--34},
   review={\MR{552175 (80k:26007)}},
}


\bib{Pet2000}{article}{
   author={Petermichl, Stefanie},
   title={Dyadic shifts and a logarithmic estimate for Hankel operators with
   matrix symbol},
   language={English, with English and French summaries},
   journal={C. R. Acad. Sci. Paris S\'er. I Math.},
   volume={330},
   date={2000},
   number={6},
   pages={455--460},
   issn={0764-4442},
   review={\MR{1756958 (2000m:42016)}},
   doi={10.1016/S0764-4442(00)00162-2},
}


\bib{PetTreVol2002}{article}{
   author={Petermichl, S.},
   author={Treil, S.},
   author={Volberg, A.},
   title={Why the Riesz transforms are averages of the dyadic shifts?},
   booktitle={Proceedings of the 6th International Conference on Harmonic
   Analysis and Partial Differential Equations (El Escorial, 2000)},
   journal={Publ. Mat.},
   date={2002}
}


\bib{RahSpe2015}{article}{
    author={Rahm, Robert},
    author={Spencer, Scott},
    title={Some Entropy Bump Conditions for Fractional Maximal and Integral Operators},
    date={2015},
    eprint={http://arxiv.org/abs/1504.05906v2}
}


\bib{Roch1993}{article}{
   author={Rochberg, Richard},
   title={NWO sequences, weighted potential operators, and Schr\"odinger
   eigenvalues},
   journal={Duke Math. J.},
   volume={72},
   date={1993},
   number={1},
   pages={187--215}
}


\bib{Saw1982}{article}{
   author={Sawyer, Eric T.},
   title={A characterization of a two-weight norm inequality for maximal
   operators},
   journal={Studia Math.},
   volume={75},
   date={1982},
   number={1},
   pages={1--11}
}


\bib{Saw1988}{article}{
   author={Sawyer, Eric T.},
   title={A characterization of two weight norm inequalities for fractional
   and Poisson integrals},
   journal={Trans. Amer. Math. Soc.},
   volume={308},
   date={1988},
   number={2},
   pages={533--545}
}


\bib{Wu1992}{article}{
   author={Wu, Si Jue},
   title={A wavelet characterization for weighted Hardy spaces},
   journal={Rev. Mat. Iberoamericana},
   volume={8},
   date={1992},
   number={3},
   pages={329--349}
}

\end{biblist}
\end{bibdiv}
\end{document}